\documentclass{amsart}%
\usepackage{amssymb}
\usepackage{amsfonts}
\usepackage{amsmath}
\usepackage[toc,page]{appendix}
\usepackage{graphicx}%
\providecommand{\U}[1]{\protect \rule{.1in}{.1in}}
\newtheorem{theorem}{Theorem}[section]

\newtheorem{axiom}{Axiom}

\newtheorem{definition}[theorem]{Definition}
\newtheorem{example}[theorem]{Example}

\newtheorem{lemma}[theorem]{Lemma}

\newtheorem{proposition}[theorem]{Proposition}
\newtheorem{remark}[theorem]{Remark}

\begin{document}
\title{On the quotient class of non-archimedean fields}
\author{Bruno Dinis}
\address[B. Dinis]{Departamento de Matem\'{a}tica, Faculdade de Ci\^{e}ncias da
Universidade de Lisboa, Portugal.}
\email{bmdinis@fc.ul.pt}
\author{Imme van den Berg}
\address[I.P. van den Berg]{Departamento de Matem\'{a}tica, Universidade de \'{E}vora, Portugal}
\email{ivdb@uevora.pt}
\thanks{The first author acknowledges the support of the Funda\c{c}\~{a}o para a
Ci\^{e}ncia e a Tecnologia, Portugal [grant SFRH/BPD/97436/2013]}
\date{}
\maketitle

\begin{abstract}
The \textit{quotient class }of a non-archimedean field is the set of cosets
with respect to all of its additive convex subgroups. The algebraic operations
on the quotient class are the Minkowski sum and product. We study the
algebraic laws of these operations. Addition and multiplication have a common
structure in terms of regular ordered semigroups. The two algebraic operations
are related by an adapted distributivity law.

\bigskip


\bigskip


\end{abstract}

\section{Introduction}

We study the algebraic properties of the set of cosets with respect to all
possible convex additive subgroups of a non-archimedean field $F$, typically a
field of formal series or a Hardy-field. We will call this set of cosets the
\emph{quotient class }of $F$. Because the (Minkowski) sum of a nontrivial
convex additive subgroup and an arbitrary element can never be zero, the
quotient class cannot be a group for addition, and for similar reasons neither
for multiplication. Still a quotient class satisfies rather strong algebraic
properties, for, as we will see, addition and multiplication are commutative,
satisfy the properties of regular semigroups and are related by an adapted
distributive law.

The common structure of addition and multiplication is stronger than a regular
semigroup and was called \emph{assembly }in \cite{dinisberg}. We will call
\emph{magnitude }a convex additive subgroup $M$ of $F$, this is in line with a
common interpretation of Hardy-fields as models of orders of magnitude of
functions \cite{Hardy}\cite{Bourbaki}\cite{Aschenbrenner Van Den Dries}. There
exists a definite relationship between non-archimedean structures and
asymptotics \cite{Bosch Remmert}\cite{debruijn}\cite{Van der Corput}%
\cite{Lightstone Robinson}; in a sense, a magnitude may be seen as the size of
an imprecision. Given a coset with respect to $M$, the magnitude $M$ acts as
an individualized neutral element for addition. If $\alpha$ is a coset which
is not reduced to a magnitude $M$ it has an individualized neutral element for
multiplication $M/\alpha$, which with some abuse of language is called
\emph{unity}.

It is easy to see that distributivity does not hold in general. However we
show that distributivity holds up to a correction term which has the form of a
magnitude. We will identify other properties which relate addition and
multiplication and call the resulting structure \emph{association}.

The order relation in the ordered field $F$ induces a natural order in the
quotient class $Q$. We show that this is a total dense order relation
compatible with the operations. If $F$ is archimedean, the quotient class
reduces to an ordered field. If $F$ is non-archimedean, the quotient class
contains magnitudes different from $\{0\}$ and its domain. Clearly $\{0\}$ is
the minimal magnitude of $Q$, but $Q$ has also a maximal magnitude which is
its domain $F$ itself; the minimal unity is $\{1\}$. In general, an
association with these properties is called a \emph{solid}. So we will prove
that a quotient class of a non-archimedean field $F$ is a solid. For the sake
of clarity we give a full list of the axioms of a solid in the appendix.

As remarked above, in solids distributivity does not hold in general. However,
it turns out that in many cases full distributivity does hold, for example for
elements of the same sign. Also it is possible to give necessary and
sufficient conditions for the distributive law to hold for triples of elements
of solids. The proofs are rather involved and are presented in a second paper.

In Section \ref{Section Quotient Classes} we define the quotient class of an
ordered field. We extend the order relation to the quotient class, prove that
the property of trichotomy is maintained and show compatibility properties of
the order with the algebraic operations. We recall also some basic notions of
semigroups. In Section \ref{Section Assemblies} we recall the notion of
assembly which amounts to a regular semigroup with an idempotent condition on
the magnitude operator. As a consequence the magnitude operator will be
linear. We show that the quotient class is an assembly for addition and,
leaving out the magnitudes, for multiplication. In Section
\ref{Section Mixed properties} we define a structure called association which
is, roughly speaking, a ring with individualized neutral elements for both
addition and multiplication, and an adapted distributive law. Ordered
associations are associations equipped with a total order relation respecting
the algebraic operations. In Section \ref{Section Solids} we define solids
which are in a sense weakly distributive ordered fields with generalized
neutral elements given by magnitudes and unities. We show that the quotient
class of a non-archimedean field is a solid.

By the above, solids arise with non-archimedean fields. Archimedean solids may
exist, but only in a set theory with a different axiomatics than conventional
set theory. This question is briefly addressed at the end of the last section.

\section{Quotient classes\label{Section Quotient Classes}}

Let $\left(  F,+,\cdot,\leq,0,1\right)  $ be a non-archimedean ordered field.
Let $\mathcal{C}$ be the set of all convex subgroups for addition of $F$ and
$Q$ be the set of all cosets with respect to the elements of $\mathcal{C}$. We
will call the elements of $\mathcal{C}$ \emph{magnitudes} and $Q$ the
\emph{quotient class} of $F$ with respect to $\mathcal{C}$. Observe that
$\mathcal{C}$ is not reduced to $\left \{  0\right \}  $ and $F$ itself. Indeed,
a non-archimedean ordered field necessarily has infinitesimals other than $0$.
Let $\oslash$ denote the set of all infinitesimals in $F$. It is clearly
convex and satisfies the group property, so $\oslash \in \mathcal{C}$. An
element of $Q\backslash \mathcal{C}$ is called \emph{zeroless}.

For $\alpha \in Q$, in the remainder of this section we use the notation
$\alpha=a+A$, with $a\in F$ and $A\in \mathcal{C}$. Clearly $A$ is unique but
$a$ is not. If $\alpha$ is zeroless, one proves by induction that
$A/a\subseteq \lbrack-1/n,1/n]$ for all $n\in \mathbb{N}$, hence $A/a\subseteq
\oslash$.

We define (with abuse of notation) addition in $Q$ pointwise, i.e. by the
Minkowski sum%
\[
\alpha+\beta:=a+b+A+B.
\]
We define (with abuse of notation) multiplication in $Q$\ also pointwise, by
\[
\alpha \beta:=ab+aB+bA+AB\text{.}%
\]
Let also $\alpha=a^{\prime}+A$\ and $\beta=b^{\prime}+B$. Then $a^{\prime
}+b^{\prime}-(a+b)\in A+B$ and
\begin{align*}
a^{\prime}b^{\prime}-ab  &  =a(b^{\prime}-b)+b(a^{\prime}-a)+(a^{\prime
}-a)(b^{\prime}-b)\\
&  \in aB+bA+AB,
\end{align*}
hence addition and multiplication do not depend on the choice of representatives.

By convexity the sum of two magnitudes $A$ and $B$ is equal to one of them,
i.e.%
\begin{equation}
A+B=A\vee A+B=B, \label{sum is one}%
\end{equation}
according to whether $B\subseteq A$ or $A\subseteq B$.

We recall that a \emph{(Von Neumann) regular} semigroup is a semigroup in
which every element is regular, i.e. for all $a\in S$ there is $x\in S$ such
that $a+x+a=a$. In this way one may think of $x$ as a "weak inverse" of $a$. A
\emph{completely regular} semigroup is a semigroup in which every element is
in some subgroup of the semigroup. We will show that $(Q,+)$ and
$(Q\backslash \mathcal{C},\cdot)$ are completely regular semigroups. We will
see later that these structures are indeed somewhat stronger.

Next lemma says that two elements of $Q$ are either separated, i.e. their
intersection is empty, or one is contained in the other.

\begin{lemma}
\label{Tricotomia}Let $\alpha,\beta \in Q$. Then%
\[
\alpha \cap \beta=\emptyset \vee \alpha \subseteq \beta \vee \beta \subseteq
\alpha \text{.}%
\]

\end{lemma}

\begin{proof}
Suppose that $\alpha \cap \beta \neq \emptyset$. Then there exists $x$ in $F$ such
that $x\in \alpha$ and $x\in \beta$. We may write $\alpha=x+A$ and $\beta=x+B$.
If $A\subseteq B$, one has $\alpha \subseteq \beta$ and if $B\subseteq A$, one
has $\beta \subseteq \alpha$.
\end{proof}

We now use the order relation $\leq$ on $F$ to define an order relation, also
noted $\leq$, on $Q$; we will see below that it is a total order relation
respecting the operations, extending the order relation on $F$.

\begin{definition}
\label{order external numbers}Given $\alpha,\beta \in Q$, we say that $\alpha$
is \emph{less than or equal to} $\beta$ and we write (with abuse of notation)
$\alpha \leq \beta$, if and only if
\begin{equation}
(\forall x\in \alpha)(\exists y\in \beta)(x\leq y). \label{def}%
\end{equation}
We say that $\alpha$ is \emph{(strictly) less than} $\beta$ and write
$\alpha<\beta$, if $\alpha \leq \beta$ and $\alpha \cap \beta=\emptyset$.
\end{definition}

Note that if $\alpha$ and $\beta$ are separated, formula (\ref{def}) is
equivalent to $(\forall x\in \alpha)(\forall y\in \beta)(x<y)$.

If the magnitudes of $\alpha$ and $\beta$ are $\{0\}$, these elements may be
identified with elements of $F$, and the order relation (\ref{def})
corresponds to the original order relation $\leq$ on $F$. Let $A$ and $B$ be
magnitudes. By convexity, $A\leq B$ if and only if $A\subseteq B$, i.e.
magnitudes are ordered by inclusion. Clearly $0\in A$ for every magnitude $A$
and as a consequence all magnitudes are positive. Since $A\subseteq B$ if and
only if $A+B=B$, the order relation on the magnitudes corresponds to the
natural partial order restricted to idempotents, see for example \cite[p.
14]{Howie}\cite[p. 18]{Petrich}.

\begin{theorem}
\label{total order rel}The relation $\leq$ is a total order relation. It is
compatible with addition and multiplication in the following way:

\begin{enumerate}
\item $\forall \alpha \forall \beta \forall \gamma \left(  \alpha \leq \beta
\Rightarrow \alpha+\gamma \leq \beta+\gamma \right)  $.

\item $\forall \alpha(A<\alpha \Rightarrow \forall \beta \forall \gamma \left(
\beta \leq \gamma \Rightarrow \alpha \beta \leq \alpha \gamma \right)  $.

\item $\forall A\forall \beta \forall \gamma \left(  B\leq \beta \leq \gamma
\Rightarrow A\beta \leq A\gamma \right)  $.
\end{enumerate}
\end{theorem}

\begin{proof}
Let $\alpha,\beta$ and $\gamma$ be arbitrary elements of $Q$. We prove firstly
that the relation $\leq$ is a total order relation on $Q$. Let $x\in \alpha$.
Because $x\leq x$ one has $\alpha \leq \alpha$, so the relation is reflexive.
Suppose that $\alpha \leq \beta$ and $\beta \leq \gamma$. Then for all $x\in
\alpha$ there exist $y\in \beta$ and $z\in \gamma$ such that $x\leq y$ and
$y\leq z$. Hence $x\leq z$ and the relation is transitive. Suppose now that
$\alpha \leq \beta$ and $\beta \leq \alpha$. Let $x\in \alpha$. Because $\alpha
\leq \beta$ there exists $y\in \beta$ such that $x\leq y$. There exists
$y^{\prime}\in \beta$ such that $y^{\prime}\leq x$, if not, all $x\in \alpha$
and $y^{\prime}\in \beta$ satisfy $x<y^{\prime}$, in contradiction with the
fact that $\beta \leq \alpha$. By convexity $x\in \lbrack y^{\prime}%
,y]\subseteq \beta$, hence $\alpha \subseteq \beta$. In an analogous way one
shows that $\beta \subseteq \alpha$. Hence $\alpha=\beta$ and the relation is
antisymmetric. To prove the totality property suppose that $\alpha \nleq \beta$.
Then there is $x\in \alpha$ such that $y<x$ for all $y\in \beta$. Hence
$\beta \leq \alpha$. We conclude that the relation $\leq$ is a total order relation.

We finish by proving the three compatibility properties.

\begin{enumerate}
\item Suppose that $\alpha \leq \beta$. Let $w\in \alpha+\gamma$. Then there are
$x\in \alpha$ and $z\in \gamma$ such that $w=x+z$. Now there exists $y\in \beta$
such that $x\leq y$. Hence $x+z\leq y+z\in \beta+\gamma$ and one concludes that
$\alpha+\gamma \leq \beta+\gamma$.

\item Suppose that $A<\alpha$ and $\beta \leq \gamma$. Let $w\in \alpha \beta$.
Then there exist $x\in \alpha$, $0<\alpha$ and $y\in \beta$ such that $w=xy$.
Because $\beta \leq \gamma$ there is $z\in \gamma$ such that $y\leq z$. Then
$xy\leq yz$, hence $\alpha \gamma \leq \beta \gamma$.

\item Suppose that $B\leq \beta \leq \gamma$. Let $w\in A\beta$. Because
$B\leq \beta$ the element $w$ may be supposed positive. Then there exist $x\in
A$, $0\leq x$ and $0\leq y\in \beta$ such that $w=xy$. Because $\beta \leq
\gamma$ there is $z\in \gamma$ such that $y\leq z$. Then $xy\leq yz$, hence
$A\gamma \leq A\gamma$.
\end{enumerate}
\end{proof}

The above proposition states that usual compatibility holds for multiplication
by strictly positive (zeroless) elements $\alpha$.

If $\alpha=A$ is a magnitude, the rule must be restricted to nonnegative
$\beta$ and $\gamma$, for instance, if $\omega>0$ is infinitely large, one has
$-\omega<-1$, while $\oslash \cdot \left(  -1\right)  \leq \oslash \cdot \left(
-\omega \right)  $, for $\oslash \cdot \left(  -1\right)  =\oslash$ and
$\oslash \cdot \left(  -\omega \right)  =\oslash \cdot \omega \geq1/\omega
\cdot \omega=1$.

\section{The magnitude operator. Assemblies.\label{Section Assemblies}}

Let $\alpha=a+A\in Q$. Then the magnitude $A$ is a sort of individualized
neutral element, since $\alpha+A=a+A+A=a+A=\alpha$. As regards to other
magnitudes $B$ which leave $\alpha$ invariant it distinguishes itself by the
property $A+B=A$ and being uniquely determined by $\alpha$. Hence we
may\ define a function $e:Q\rightarrow \mathcal{C}$ by putting $e\left(
\alpha \right)  =A$. The function is linear, for if $\beta=b+B\in Q$%
\begin{equation}
e\left(  \alpha+\beta \right)  =A+B=e\left(  \alpha \right)  +e\left(
\beta \right)  . \label{magnitude linear}%
\end{equation}
Also, by (\ref{sum is one})%
\begin{equation}
e\left(  \alpha+\beta \right)  =e\left(  \alpha \right)  \vee e(\alpha
+\beta)=e\left(  \beta \right)  . \label{magnitude linear 2}%
\end{equation}
With respect to $e\left(  \alpha \right)  $ we may also identify a
distinguished symmetrical element $s(\alpha)$, having the same magnitude as
$\alpha$, simply by putting $s(\alpha)=-\alpha=-a+A$. Semigroup structures
with the above properties for individualized neutral and symmetrical elements
have been called \emph{assemblies} in \cite{dinisberg}.

Below we list the axioms of an assembly. It is easy to verify that the element
$e$ of Definition \ref{AssemblyA}.\ref{assemblyneut copy(1)} is unique (see
Remark \ref{e is function}) and with some abuse of language we use the same
notation as above.

\begin{definition}
\label{AssemblyA}A non-empty structure $(\mathcal{A},+)$ is called an
\emph{assembly} if

\begin{enumerate}
\item \label{assemblyassoc copy(1)}$\forall x\forall y\forall z(x+\left(
y+z\right)  =\left(  x+y\right)  +z).$

\item \label{assemblycom copy(1)}$\forall x\forall y(x+y=y+x).$

\item \label{assemblyneut copy(1)}$\forall x\exists e\left(  x+e=x\wedge
\forall f\left(  x+f=x\rightarrow e+f=e\right)  \right)  .$

\item \label{assemblysim copy(1)}$\forall x\exists s\left(  x+s=e\left(
x\right)  \wedge e\left(  s\right)  =e\left(  x\right)  \right)  .$

\item \label{assemblye(xy) copy(1)}$\forall x\forall y\left(  e\left(
x+y\right)  =e\left(  x\right)  \vee e\left(  x+y\right)  =e\left(  y\right)
\right)  .$
\end{enumerate}
\end{definition}

We will prove that $(Q,+)$ and $(Q\backslash \mathcal{C},\cdot)$ are
assemblies. First we prove that they are completely regular commutative
semigroups. Let $A$ be a magnitude. We let $F_{A}=\left \{  x+A|x\in F\right \}
$ and $R_{A}=\left \{  x\left(  1+A/a\right)  |x\in F\backslash \left \{
0\right \}  \right \}  $.

\begin{proposition}
$(Q,+)$ is a completely regular commutative semigroup.
\end{proposition}

\begin{proof}
Clearly addition is associative and commutative. Let $\alpha=a+A\in Q$. Then
$\alpha \in F_{A}$. Observe that $e\left(  \alpha \right)  $ is the neutral
element of $F_{A}$ and $s(\alpha)$ is the symmetric element of $\alpha$ in
$F_{A}$, hence $F_{A}$ is a subgroup of $Q$. Hence the commutative semigroup
$(Q,+)$ is completely regular.
\end{proof}

\begin{proposition}
\label{QC semigroup}$(Q\backslash \mathcal{C},\cdot)$ is a completely regular
commutative semigroup.
\end{proposition}

\begin{proof}
Clearly multiplication is associative and commutative. Let $\alpha=a+A\in
Q\backslash \mathcal{C}$. We may write $\alpha=a\left(  1+A/a\right)  $. Recall
that $1+A/a$ is zeroless. Then $\alpha \in R_{A}$, which is a subgroup of $Q$;
this follows from the fact that $\left(  1+A/a\right)  \left(  1+A/a\right)
=1+A/a+A/a+A/a\cdot A/a=\left(  1+A/a\right)  $, noting that $A/a\cdot
A/a\subseteq A/a\cdot1=A/a$. Hence the commutative semigroup $(Q\backslash
\mathcal{C},\cdot)$ is completely regular.
\end{proof}

The above structures are not proper groups. Indeed, if $X\in \mathcal{C},X>0$,
the equation $1+\oslash+(-1+X)=0$ does not have a solution, for $\oslash+X\geq
X>0$. Also $(1+X)/(1+\oslash)\neq1$ for any convex group $X$, for
$(1+X)/(1+\oslash)=(1+X)(1+\oslash)=1+X+\oslash+X\oslash>1$.

An assembly is a completely regular commutative semigroup, for every
$a\in \mathcal{A}$ is element of the group $\mathcal{A}_{e(a)}=$ $\left \{
x\in \mathcal{A}|e(x)=e(a)\right \}  $. Conversely, a completely regular
commutative semigroup is an assembly with $s\left(  \alpha \right)  =-a+A$ if
the operator $e$ satisfies condition (\ref{assemblye(xy) copy(1)}) of
Definition \ref{AssemblyA}. By the remarks above and formula
(\ref{magnitude linear 2}), we have the following proposition.

\begin{proposition}
\label{Q assembleia}The structure $(Q,+)$ is an assembly.
\end{proposition}

The proposition is also true for the multiplicative structure $(Q\backslash
\mathcal{C},\cdot)$. Let $u:Q\backslash \mathcal{C\rightarrow}Q\backslash
\mathcal{C}$ be defined by $u(\alpha)=1+A/\alpha$ and let $d:Q\backslash
\mathcal{C\rightarrow}Q\backslash \mathcal{C}$ be defined by $d\left(
\alpha \right)  =\frac{1}{\alpha}$. Note that for all $a$ such that
$\alpha=a+A$
\begin{equation}
\frac{1}{\alpha}=\frac{1}{a+A}=\frac{1}{a(1+A/a)}=\frac{1}{a}\left(
1+\frac{A}{a}\right)  =\frac{1}{a}+\frac{A}{a^{2}}\in Q\backslash \mathcal{C}.
\label{formule1alfa}%
\end{equation}
Then also for all $a$ such that $\alpha=a+A$
\begin{equation}
\frac{A}{\alpha}=\frac{A}{a}+\frac{A^{2}}{a^{2}}=\frac{A}{a}
\label{formuleGalfa}%
\end{equation}
and $u(\alpha)=1+A/a\in Q\backslash \mathcal{C}$. This means that the functions
$u$ and $d$ are well-defined. Clearly $u\left(  \alpha \right)  $ is the
multiplicative neutral element (unity) of $\alpha$ and $d\left(
\alpha \right)  $ is the inverse of $\alpha$ in the group $R_{A}$. In
particular $\alpha u\left(  \alpha \right)  =\alpha$ and $\alpha d\left(
\alpha \right)  =u\left(  \alpha \right)  $. Note that
\begin{equation}
e\left(  u\left(  \alpha \right)  \right)  =e\left(  \alpha \right)
/\alpha \label{e(u(alpha))}%
\end{equation}
and
\begin{equation}
e\left(  d\left(  \alpha \right)  \right)  =e\left(  \alpha \right)  /\alpha
^{2}. \label{e(d(alpha))}%
\end{equation}

\begin{proposition}
\label{Q* assembleia}The structure $(Q\backslash \mathcal{C},\cdot)$ is an assembly.
\end{proposition}

\begin{proof}
By Proposition \ref{QC semigroup} we need only to verify condition
(\ref{assemblye(xy) copy(1)}) of Definition \ref{AssemblyA}. Let $\alpha
,\beta \in Q\backslash \mathcal{C}$. Then $aB+bA+AB=aB+bA$. Hence
\[
u\left(  \alpha \beta \right)  =u\left(  ab+aB+bA+AB\right)  =1+\frac{aB+bA}%
{ab}=1+\frac{B}{b}+\frac{A}{a},
\]
which is equal to $u(\alpha)$ or $u(\beta)$.
\end{proof}

Classical models of orders of magnitude are based on the $O^{\prime}s$ and
$o^{\prime}s$. They can be seen as sets of real functions, for which addition
can be defined pointwise \cite{debruijn}. We give an example where $O^{\prime
}s$ and $o^{\prime}s$ give rise to additive and multiplicative assemblies.
This is done in the context of a non-archimedean field, in which all the
magnitudes except $\left \{  0\right \}  $ and the field itself may be
determined in terms of $O^{\prime}s$ and $o^{\prime}s$; in fact, the
$o^{\prime}s$ are reduced to $O^{\prime}s$. Let $\mathcal{R}$ be the set of
all rational fractions with coefficients in $\mathbb{R}$ with the usual
addition and multiplication. Let $n\in \mathbb{Z}$. Clearly $O\left(
x^{n}\right)  $ is a magnitude, for $x\rightarrow \infty$ and then $o\left(
x^{n}\right)  =O\left(  x^{n-1}\right)  $. Conversely, let $\left \{
0\right \}  \subset M\subset \mathcal{R}$ be a magnitude. Let $n\in \mathbb{N}$
be minimal such that $x^{n}\notin M$. If there exists $r\in \mathcal{R}$ such
that $\limsup r\left(  x\right)  /x^{n-1}=\lim r\left(  x\right)
/x^{n-1}=\infty$ then the degree of $r$ is equal to $n$. Hence $x^{n}\in M$, a
contradiction. Hence $M=O\left(  x^{n-1}\right)  $. Let $Q$ be the quotient
field of $\mathcal{R}$. So within $Q$ the $O^{\prime}s$ define additive and
multiplicative assemblies.

The following example shows that $O^{\prime}s$ and $o^{\prime}s$ do not
generate assemblies in general.

\begin{example}
\label{example not assembly}\emph{We will show that condition
(\ref{assemblye(xy) copy(1)}) of Definition \ref{AssemblyA} does not hold for
}$O^{\prime}s$\emph{ and }$o^{\prime}s$\emph{ of real functions. Let
}$f,g:R\rightarrow R^{+}$\emph{ be defined by }$f(x)=x$\emph{ }$+x^{2}(\sin
x,0)^{+}$\emph{ and }$g(x)=x$\emph{ }$+x^{2}(\cos x,0)^{+}$\emph{. For
}$x\rightarrow+\infty$\emph{ we have }$O(f+g)=O(x^{2})$\emph{, but since
}$x^{2}\notin O(f)$\emph{ and }$x^{2}\notin O(g)$\emph{, neither
}$O(f)=O(f+g)$\emph{, nor }$O(g)=O(f+g)$\emph{. For the same reason neither
}$o(f)=o(f+g)$\emph{, nor }$o(g)=o(f+g)$\emph{.}
\end{example}

We end with examples of assemblies in a different context.

\begin{example}
\label{exampleassembly}

\begin{enumerate}
\item Commutative groups are assemblies on which the function $e$ is constant.

\item Let $C$ be a chain for inclusion with the union operation $\cup$. The
structure $(C,\cup)$ is an assembly, with $e(U)=s(U)=U$ for all $U\subseteq
C$. Note that $e(U\cup V)=U\cup V$. Hence $e(U\cup V)=e(U)$ or $e(U\cup
V)=e(V)$.
\end{enumerate}
\end{example}

\section{Mixed properties of addition and
multiplication\label{Section Mixed properties}}

We will see by a simple example that distributivity does not hold in $Q$.
Still an adapted version of distributivity does hold, which requires the
introduction of a correcting term in the form of a magnitude. Then we
calculate the magnitude and the symmetrical of the product. We introduce the
notion of association which roughly speaking stays in relation to rings in the
way assemblies are to groups. Associations with a total order relation
compatible with the operations are called ordered associations. Finally we
show that $Q$ is indeed an ordered association.

We start by showing that distributivity does not hold in $Q$.

\begin{example}
$0=\oslash(1-1)\neq \oslash1-\oslash1=\oslash$.
\end{example}

In the example the error made has the form of a magnitude. This is generally
true and follows from the next two propositions. The first proposition gives
the form of the error term and the second one shows that this error is a magnitude.

\begin{proposition}
\label{formula dist total cosets}Let $\alpha=a+A,\beta=b+B,\gamma=c+C\in Q$.
Then $\alpha \beta+\alpha \gamma=\alpha \left(  \beta+\gamma \right)
+A\beta+A\gamma$.
\end{proposition}

\begin{proof}
Because $F$ is a field $a\left(  b+c\right)  =ab+ac$. Furthermore $A\left(
b+c\right)  \subseteq bA+cA$, because $\left \vert b+c\right \vert \leq
2\max(\left \vert b\right \vert ,\left \vert c\right \vert )$, $2A=A$ and
$bA+cA=\max(\left \vert b\right \vert ,\left \vert c\right \vert )A$. Also we have
the identity of groups $A\left(  B+C\right)  =AB+AC$. Hence%
\begin{align*}
&  \alpha \left(  \beta+\gamma \right)  +A\beta+A\gamma \\
&  =\left(  a+A\right)  \left(  b+c+B+C\right)  +A(b+B)+A(c+C)\\
&  =a\left(  b+c\right)  +a\left(  B+C\right)  +A\left(  b+c\right)  +A\left(
B+C\right)  +bA+AB+cA+AC\\
&  =ab+ac+aB+aC+bA+cA+AB+AC\\
&  =\left(  a+A\right)  b+\left(  a+A\right)  B+\left(  a+A\right)  c+\left(
a+A\right)  C\\
&  =\left(  a+A\right)  \left(  b+B\right)  +\left(  a+A\right)  \left(
c+C\right) \\
&  =\alpha \beta+\alpha \gamma \text{.}%
\end{align*}

\end{proof}

Next proposition shows that the correction term in the adapted version of
distributivity is a magnitude.

\begin{proposition}
\label{e(a)b=d}Let $\alpha=a+A,\beta=b+B\in Q$. Then there exists $\delta \in
Q$ such that $e\left(  \alpha \right)  \beta=e\left(  \delta \right)  $.
\end{proposition}

\begin{proof}
Put $\delta=bA+AB$. One has%
\[
e\left(  \alpha \right)  \beta=A\left(  b+B\right)  =bA+AB=\delta=e\left(
\delta \right)  \text{.}%
\]

\end{proof}

It is not difficult to determine the magnitudes of a product, the unity
element and the inverse of a zeroless element in $Q$. In fact we have the
following proposition.

\begin{proposition}
\label{Prop Archim magnitude product}Let $\alpha,\beta \in Q$. Then

\begin{enumerate}
\item \label{Magnitude product}$e(\alpha \beta)=\alpha e(\beta)+\beta
e(\alpha)$.

\item \label{s(ab)=as(b)}$-\left(  \alpha \beta \right)  =\left(  -\alpha
\right)  \beta=\alpha \left(  -\beta \right)  $.
\end{enumerate}
\end{proposition}

\begin{proof}
Let $\alpha=a+A,\beta=b+B\in Q$.

\ref{Magnitude product}. One has $\beta e\left(  \alpha \right)  +\alpha
e\left(  \beta \right)  =A\left(  b+B\right)  +\left(  a+A\right)
B=bA+AB+aB+AB=e\left(  ab+bA+aB+AB\right)  =e\left(  \alpha \beta \right)  .$

\ref{s(ab)=as(b)}. This is evident, because $-\left(  ab+bA+aB+AB\right)
=\left(  -a+A\right)  \left(  b+B\right)  =\left(  a+A\right)  \left(
-b+B\right)  $.
\end{proof}

Structures with the properties given by Proposition
\ref{formula dist total cosets}, \ref{e(a)b=d} and
\ref{Prop Archim magnitude product} and formulas (\ref{e(u(alpha))}) and
(\ref{e(d(alpha))}) will be called associations. Let $\mathcal{A}$ be an
assembly. We denote by $\mathcal{N}$ the set of all elements of $\mathcal{A}$
which are not zeroless.

\begin{definition}
\label{Association A}A structure $(\mathcal{A},+,\cdot)$ is called an
\emph{association} if the structures $(\mathcal{A},+)$ and $(\mathcal{A}%
\backslash \mathcal{N},\cdot)$ are both assemblies and if the following hold:

\begin{enumerate}
\item \label{distributivity}$\forall x\forall y\forall z\left(  xy+xz=x\left(
y+z\right)  +e\left(  x\right)  y+e\left(  x\right)  z\right)  .$

\item \label{Escala}$\forall x\forall y\exists z(e(x)y=e(z)).$

\item \label{e(xy)=e(x)y+e(y)x}$\forall x\forall y\left(
e(xy)=e(x)y+e(y)x\right)  .$

\item \label{e(u)=e(x)d(x)}$\forall x\neq e(x)\left(  e(u(x))=e(x)d(x)\right)
.$

\item \label{s(xy)=s(x)y}$\forall x\forall y\left(  s(xy)=s(x)y\right)  .$
\end{enumerate}
\end{definition}

\begin{theorem}
\label{Q Association}The structure $\left(  Q,+,e,s,\cdot,u,d\right)  $ is an association.
\end{theorem}

\begin{proof}
Directly from Proposition \ref{Q assembleia}, Proposition \ref{Q* assembleia},
formulas (\ref{e(u(alpha))}) and (\ref{e(d(alpha))}), Proposition
\ref{formula dist total cosets}, Proposition \ref{e(a)b=d} and Proposition
\ref{Prop Archim magnitude product}.
\end{proof}

Theorem \ref{total order rel} shows that it is possible to introduce in $Q$ a
total order compatible with the operations. We also have to specify the order
relation between the magnitudes and general elements. We already saw that if
$A,B$ are magnitudes such that $A+B=A,$ then $B\leq A$. Condition
(\ref{e(x)maiory}) of the next definition generalizes this to arbitrary
elements. Structures satisfying the above properties are called ordered associations.

\begin{definition}
\label{Definition Ordered assembly}We say that a structure $\left(
\mathcal{A},+,\cdot,\leq \right)  $ is an \emph{ordered association} if
$\left(  \mathcal{A},+,\cdot \right)  $ is an association, $\leq$ is a total
order relation and the following hold:

\begin{enumerate}
\item \label{comp addition}$\forall x\forall y\forall z\left(  x\leq
y\Rightarrow x+z\leq y+z\right)  .$

\item \label{e(x)maiory}$\forall x\forall y\left(  y+e(x)=e(x)\Rightarrow
y\leq e(x)\wedge s\left(  y\right)  \leq e\left(  x\right)  \right)  .$

\item \label{comp mult}$\forall x\forall y\forall z\left(  \left(  e\left(
x\right)  <x\wedge y\leq z\right)  \Rightarrow xy\leq xz\right)  .$

\item \label{Amplification}$\forall x\forall y\forall z\left(  \left(
e\left(  y\right)  \leq y\leq z\right)  \Rightarrow e\left(  x\right)  y\leq
e\left(  x\right)  z\right)  .$
\end{enumerate}
\end{definition}

\begin{theorem}
\label{Q ordered association}The structure $\left(  Q,+,e,s,\cdot
,u,d,\leq \right)  $ is an ordered association.
\end{theorem}

\begin{proof}
By Theorem \ref{total order rel} and Theorem \ref{Q Association} we only need
to show that condition (\ref{e(x)maiory}) is satisfied. Let $\alpha
=a+A,\beta=b+B\in Q$. Assume that $\alpha+B=B$, i.e. $a+A+B=B$. Then
$A\subseteq B$, so $a+A\subseteq B$. Hence $\alpha \leq B$. Also $\alpha
+\alpha \subseteq B+B=B$. Then $-\alpha=\alpha-\left(  \alpha+\alpha \right)
\subseteq B-B=B$. Hence $-\alpha \leq B$.
\end{proof}

\section{Solids\label{Section Solids}}

Clearly rings with unity are associations and the same is true for fields. As
will be shown below associations with a unique magnitude are fields.

In order to distinguish fields from associations we will postulate the
existence of particular elements. The resulting structure will be called a
solid. We finish by proving that the quotient class of a non-archimedean field
is a solid.

\begin{proposition}
Let $\left(  \mathcal{A},+,\cdot \right)  $ be an association with a unique
magnitude $e$ then $\left(  \mathcal{A},+,\cdot \right)  $ is a ring.
Furthermore if it has a unique unity $u$ then $\left(  \mathcal{A}%
,+,\cdot,e,u\right)  $ is a field.
\end{proposition}

\begin{proof}
With respect to the first part we only need to show that the magnitude is the
neutral element and that distributivity holds. Let $x\in \mathcal{A}$ and let
$e$ be the unique magnitude in $\mathcal{A}$. Then $x=x+e\left(  x\right)
=x+e$. Hence $e$ is the neutral element for addition. Observe that $ex=e$, by
condition (\ref{Escala})\ of Definition \ref{Association A}. To prove
distributivity let $x,y,z\in \mathcal{A}$. Then $xy+xz=x\left(  y+z\right)
+e\left(  x\right)  y+e\left(  x\right)  z=x\left(  y+z\right)  +e+e=x\left(
y+z\right)  $. Hence $\left(  \mathcal{A},+,\cdot \right)  $ is a ring.

To prove the second part note that $u\neq e$ because $u$ is not a magnitude.
As above, $\left(  \mathcal{A}\backslash \mathcal{N},\cdot,u\right)  $ is a
group. Hence $\left(  \mathcal{A},+,\cdot,e,u\right)  $ is a field.
\end{proof}

\begin{definition}
\label{Definition solid}A structure $(S,+,\cdot,\leq)$ is called a
\emph{solid} if $(S,+,\cdot,\leq)$ is an ordered association such that the
following hold:

\begin{enumerate}
\item \label{neut min}$\exists m\forall x\left(  m+x=x\right)  .$

\item \label{neut max}$\exists M\forall x(e\left(  x\right)  +M=M).$

\item \label{neut mult}$\exists u\forall x\left(  ux=x\right)  .$

\item \label{decomposition}$\forall x\exists a\left(  x=a+e\left(  x\right)
\wedge e\left(  a\right)  =0\right)  .$

\item \label{existencia magnitudes}$\exists x\left(  e\left(  x\right)  \neq
m\wedge e\left(  x\right)  \neq M\right)  .$

\item \label{scheiding neutrices}$\forall x\forall y(x=e\left(  x\right)
\wedge y=e(y)\wedge x<y\rightarrow \exists z(z\neq e(z)\wedge x<z<y)).$
\end{enumerate}
\end{definition}

Conditions (\ref{neut min}) and (\ref{neut mult}) are completion properties in
the sense that they postulate the existence of (minimal) neutral elements for
addition and multiplication (corresponding to $0$ and $1$ in groups and
fields). Condition (\ref{neut max}) postulates the existence of a maximal
individualized neutral element (denoted $M$). The existence of such an
absorber is a common procedure in semigroups where it is called "zero element"
(see for example \cite[p. 2]{Howie}). In the case of the structure $\left(
Q,+,e,s,\cdot,u,d,\leq \right)  $ it is the field $F$ which is the largest
magnitude. Condition (\ref{decomposition}) allows to decompose each element in
terms of an element with minimal neutral element ("precise element") and an
individualized neutral element, like the representation $\alpha=a+A$ in $Q$.
We may identify $a$ with an element of $F$. Condition
(\ref{existencia magnitudes}) postulates the existence of nontrivial neutral
elements, i.e. neutral elements besides $m$ and $M$ and has as a consequence
that effectively solids have a richer structure than fields. Condition
(\ref{scheiding neutrices}) avoids "gaps" in the sense that two magnitudes are
separated by an element which is not a magnitude.

\begin{theorem}
\label{Q solid}The structure $\left(  Q,+,\cdot,\leq \right)  $ is a solid.
\end{theorem}

\begin{proof}
By Theorem \ref{Q ordered association} we only need to verify that conditions
(\ref{neut min})-(\ref{existencia magnitudes}) of Definition
\ref{Definition solid} are satisfied.

Condition (\ref{decomposition}) is satisfied by construction. Conditions
(\ref{neut min})-(\ref{neut mult}) are satisfied taking $m=\left \{  0\right \}
$, $M=F$ and $u=\left \{  1\right \}  $. A non-archimedean ordered field
necessarily has infinitesimals other than $0$. Let $\oslash$ denote the set of
all infinitesimals in $F$. It is clearly convex and satisfies the group
property. Also $\oslash \neq \left \{  0\right \}  $ and $\oslash \neq F$, so
condition (\ref{existencia magnitudes}) also holds. To show that condition
(\ref{scheiding neutrices}) holds let $A,B$ be magnitudes in $Q$ such that
$A\subset B$. Let $b\in B\backslash A$. Then $A<b<B$.
\end{proof}

\begin{remark}
\emph{Due to the existence of non-trivial magnitudes, within ordinary set
theory }$ZFC$\emph{ any solid must be non-archimedean. Indeed, let }$x$\emph{
be such that }$0<e(x)<M$\emph{. By Definition \ref{Definition solid}%
.\ref{scheiding neutrices} there exist }$y$\emph{ such that }$e(x)<y<M$\emph{.
Then }$e^{\prime}(x)\equiv e(x)/y<1$\emph{. Now }$e^{\prime}(x)+e^{\prime
}(x)=e^{\prime}(x)$\emph{, and because the induction scheme holds, one obtains
that }$ne^{\prime}(x)=e^{\prime}(x)$\emph{ for all }$n\in \mathbb{N}$\emph{. As
a consequence }$ne^{\prime}(x)<1$\emph{ for all }$n\in \mathbb{N}$\emph{.
However, there exists also an Archimedean field such that the quotient class
with respect to its magnitudes is a solid. Such a solid exists within the
axiomatic approach to Nonstandard Analysis }$IST$\emph{\ (Internal Set Theory)
of Nelson \cite{NelsonIST}. In this approach the set of all real numbers
}$\mathbb{R}$\emph{\ is Archimedean, the axiomatics distinguishes "standard"
natural numbers and "nonstandard" natural numbers within }$\mathbb{N}$\emph{,
the latter numbers being always larger than the first. Then there exist (many)
convex ordered groups within }$\mathbb{R}$\emph{\ which are not reduced to
}$\left \{  0\right \}  $\emph{\ and }$\mathbb{R}$\emph{\ itself, like the set
of all infinitesimals. It has to be noted that they are "external sets" in the
sense of the extended axiomatics }$HST$\emph{\ presented in
\cite{KanoveiReeken}. They were called "(scalar) neutrices" in
\cite{koudjetivandenberg}, after the functional neutrices of Van der Corput
\cite{Van der Corput}. In \cite{koudjetivandenberg} a (mostly external) coset
with respect to a convex ordered subgroup within }$\mathbb{R}$\emph{\ was
called "external number". The external set of all possible external numbers
was shown to be an assembly for addition in \cite{dinisberg} and a solid in
\cite{Dinis}.}
\end{remark}

The existence of solids in different settings suggests that it is worthwhile
to investigate the algebraic properties of solids. In particular we are able
to give necessary and sufficient conditions for distributivity to hold. The
proof is rather involved and requires a thorough investigation in the algebra
of magnitudes. These results are presented in \cite{dinisberg2}.

\appendix

\section{List of axioms}

\label{Section Axioms}

The first and second group of axioms are the algebraic laws of an additive,
respectively multiplicative, assembly. The third group of axioms states that
there is a total order relation compatible with addition and multiplication,
with some particular rules for the magnitudes. The fourth group of axioms
connects addition and multiplication, together with the first three groups
they give the algebraic laws of an ordered association. The fifth group
permits to distinguish solids from associations, by postulating the existence
of particular elements: minimal neutral elements for addition and
multiplication, a maximal neutral element for addition, a decomposition,
nontrivial magnitudes and finally elements separating two magnitudes. The
axioms are written in the first-order language $L=\left \{  +,\cdot
,\leq \right \}  $.

\begin{remark}
\label{e is function}\emph{The functional notation for magnitudes is justified
by the fact that the element }$e$\emph{ of Axiom \ref{assemblyneut} is unique.
Indeed, if }$e^{\prime}$\emph{ satisfies Axiom \ref{assemblyneut}, one has
}$e^{\prime}=e^{\prime}+e=e+e^{\prime}=e$\emph{. Also }$s$\emph{ is unique and
may be considered functional. Indeed, if }$s^{\prime}$\emph{ satisfies Axiom
\ref{assemblysim} one has }$s^{\prime}=s^{\prime}+e(s^{\prime})=s^{\prime
}+e(x)=s^{\prime}+x+s=x+s^{\prime}+s=e(x)+s=e(s)+s=s$\emph{. In fact we will
use the notation }$-x$\emph{ for }$s\left(  x\right)  $\emph{. The functional
notation for unities is justified in an analogous way where we will use }%
$/x$\emph{ instead of }$d\left(  x\right)  .$\emph{ }
\end{remark}

\begin{enumerate}
\item \textbf{Axioms for addition}

\begin{axiom}
\label{assemblyassoc}$\forall x\forall y\forall z(x+\left(  y+z\right)
=\left(  x+y\right)  +z).$
\end{axiom}

\begin{axiom}
\label{assemblycom}$\forall x\forall y(x+y=y+x).$
\end{axiom}

\begin{axiom}
\label{assemblyneut}$\forall x\exists e\left(  x+e=x\wedge \forall f\left(
x+f=x\rightarrow e+f=e\right)  \right)  .$
\end{axiom}

\begin{axiom}
\label{assemblysim}$\forall x\exists s\left(  x+s=e\left(  x\right)  \wedge
e\left(  s\right)  =e\left(  x\right)  \right)  .$
\end{axiom}

\begin{axiom}
\label{assemblye(xy)}$\forall x\forall y\left(  e\left(  x+y\right)  =e\left(
x\right)  \vee e\left(  x+y\right)  =e\left(  y\right)  \right)  .$
\end{axiom}

\item \textbf{Axioms for multiplication}

\begin{axiom}
\label{axiom assoc mult}$\forall x\forall y\forall z(x\left(  yz\right)
=\left(  xy\right)  z).$
\end{axiom}

\begin{axiom}
\label{axiom com mult}$\forall x\forall y(xy=yx).$
\end{axiom}

\begin{axiom}
\label{axiom neut mult}$\forall x\neq e\left(  x\right)  \exists u\left(
xu=x\wedge \forall v\left(  xv=x\rightarrow uv=u\right)  \right)  .$
\end{axiom}

\begin{axiom}
\label{axiom sym mult}$\forall x\neq e\left(  x\right)  \exists d\left(
xd=u\left(  x\right)  \wedge u\left(  d\right)  =u\left(  x\right)  \right)
.$
\end{axiom}

\begin{axiom}
\label{axiom u(xy)}$\forall x\neq e\left(  x\right)  \forall y\neq e\left(
y\right)  \left(  u\left(  xy\right)  =u\left(  x\right)  \vee u\left(
xy\right)  =u\left(  y\right)  \right)  .$
\end{axiom}

\item \textbf{Order axioms}

\begin{axiom}
\label{(OA)reflex}$\forall x(x\leq x).$
\end{axiom}

\begin{axiom}
\label{(OA)antisym}$\forall x\forall y(x\leq y\wedge y\leq x\rightarrow x=y).
$
\end{axiom}

\begin{axiom}
\label{(OA)trans}$\forall x\forall y\forall z(x\leq y\wedge y\leq z\rightarrow
x\leq z).$
\end{axiom}

\begin{axiom}
\label{(OA)total}$\forall x\forall y(x\leq y\vee y\leq x).$
\end{axiom}

\begin{axiom}
\label{(OA)compoper}$\forall x\forall y\forall z\left(  x\leq y\rightarrow
x+z\leq y+z\right)  .$
\end{axiom}

\begin{axiom}
\label{Axiom e(x)maiory}$\forall x\forall y\left(  y+e(x)=e(x)\rightarrow
\left(  y\leq e(x)\wedge-y\leq e(x)\right)  \right)  .$
\end{axiom}

\begin{axiom}
\label{compat mult}$\forall x\forall y\forall z\left(  \left(  e\left(
x\right)  <x\wedge y\leq z\right)  \rightarrow xy\leq xz\right)  .$
\end{axiom}

\begin{axiom}
\label{Axiom Amplification}$\forall x\forall y\forall z\left(  \left(
e\left(  y\right)  \leq y\leq z\right)  \rightarrow e\left(  x\right)  y\leq
e\left(  x\right)  z\right)  .$
\end{axiom}

\item \textbf{Axioms relating addition and multiplication}

\begin{axiom}
\label{Axiom escala}$\forall x\forall y\exists z(e(x)y=e(z)).$
\end{axiom}

\begin{axiom}
\label{Axiom e(xy)=e(x)y+e(y)x}$\forall x\forall y\left(
e(xy)=e(x)y+e(y)x\right)  .$
\end{axiom}

\begin{axiom}
\label{e(u(x))=e(x)d(x)}$\forall x\neq e(x)\left(  e(u(x))=e(x)/x\right)  .$
\end{axiom}

\begin{axiom}
\label{Axiom distributivity}$\forall x\forall y\forall z\left(  xy+xz=x\left(
y+z\right)  +e\left(  x\right)  y+e\left(  x\right)  z\right)  .$
\end{axiom}

\begin{axiom}
\label{Axiom s(xy)=s(x)y}$\forall x\forall y\left(  -(xy)=(-x)y\right)  .$
\end{axiom}

\item \textbf{Axioms of existence}

\begin{axiom}
\label{Axiom neut min}$\exists m\forall x\left(  m+x=x\right)  .$
\end{axiom}

\begin{axiom}
\label{Axiom neut mult}$\exists u\forall x\left(  ux=x\right)  .$
\end{axiom}

\begin{axiom}
\label{Axiom neut max}$\exists M\forall x(e\left(  x\right)  +M=M).$
\end{axiom}

\begin{axiom}
\label{existencia neutrices}$\exists x\left(  e\left(  x\right)  \neq0\wedge
e\left(  x\right)  \neq M\right)  .$
\end{axiom}

\begin{axiom}
\label{numexterno}$\forall x\exists a\left(  x=a+e\left(  x\right)  \wedge
e\left(  a\right)  =0\right)  .$
\end{axiom}

\begin{axiom}
\label{scheiding neutrices(1)}$\forall x\forall y(x=e\left(  x\right)  \wedge
y=e(y)\wedge x<y\rightarrow \exists z(z\neq e(z)\wedge x<z<y)).$
\end{axiom}
\end{enumerate}

\end{document}